\documentclass[english]{amsart}
\usepackage[T1]{fontenc}
\usepackage[latin9]{inputenc}
\usepackage{amstext}
\usepackage{amsthm}
\usepackage{amssymb}

\makeatletter
\numberwithin{equation}{section}
\numberwithin{figure}{section}
\theoremstyle{plain}
\newtheorem{thm}{\protect\theoremname}[section]
\theoremstyle{definition}
\newtheorem{defn}[thm]{\protect\definitionname}
\theoremstyle{plain}
\newtheorem{fact}[thm]{\protect\factname}
\theoremstyle{plain}
\newtheorem{lem}[thm]{\protect\lemmaname}

\makeatother

\usepackage{babel}
\providecommand{\definitionname}{Definition}
\providecommand{\factname}{Fact}
\providecommand{\lemmaname}{Lemma}
\providecommand{\theoremname}{Theorem}

\begin{document}

\title{Combined Algebraic Properties in Gaussian and Quaternion Ring}

\author{Aninda Chakraborty}

\address{Department of Mathematics, Government General Degree College at Chapra,
University of Kalyani}

\email{anindachakraborty2@gmail.com}
\begin{abstract}
It is known that for an IP$^{*}$ set $A$ in $(\mathbb{N},+)$ and
a sequence $\left\langle x_{n}\right\rangle _{n=1}^{\infty}$ in $\mathbb{N}$,
there exists a sum subsystem $\left\langle y_{n}\right\rangle _{n=1}^{\infty}$
of $\left\langle x_{n}\right\rangle _{n=1}^{\infty}$ such that $FS\left(\left\langle y_{n}\right\rangle _{n=1}^{\infty}\right)\cup FP\left(\left\langle y_{n}\right\rangle _{n=1}^{\infty}\right)\subseteq A$.
Similar types of results have also been proved for central$^{*}$
sets and $C^{*}$-sets where the sequences have been considered from
the class of minimal sequences and almost minimal sequences. In this
present work, our aim to establish the similar type of results for
the ring of Gaussian integers and the ring of integer quaternions.
\end{abstract}

\maketitle

\section{Introduction}

A famous Ramsey-theoretic result is Hindman\textquoteright s Theorem: 
\begin{thm}
Given a finite coloring of $\mathbb{N}=\bigcup_{i=1}^{r}A_{i}$, there
exists a sequence $\left\langle x_{n}\right\rangle _{n=1}^{\infty}$
in $\mathbb{N}$ and $i\in\left\{ 1,2,\ldots,r\right\} $ such that
\[
FS\left(\left\langle x_{n}\right\rangle _{n=1}^{\infty}\right)=\left\{ \sum_{n\in F}x_{n}:F\in\mathcal{P}_{f}\left(\mathbb{N}\right)\right\} \subseteq A_{i},
\]
 where for any set $X$, $\mathcal{P}_{f}\left(X\right)$ is the set
of all finite nonempty subsets of $X$.
\end{thm}

A strongly negative answer to a combined additive and multiplicative
version of Hindman\textquoteright s Theorem was presented in \cite[Theorem 2.11]{key-6}.
Given a sequence $\left\langle x_{n}\right\rangle _{n=1}^{\infty}$
in $\mathbb{N}$, Let 
\[
PS\left(\left\langle x_{n}\right\rangle _{n=1}^{\infty}\right)=\left\{ x_{m}+x_{n}:m,n\in\mathbb{N},m\neq n\right\} 
\]
 and 
\[
PP\left(\left\langle x_{n}\right\rangle _{n=1}^{\infty}\right)=\left\{ x_{m}\cdot x_{n}:m,n\in\mathbb{N},m\neq n\right\} .
\]

\begin{thm}
There exists a finite partition $\mathcal{R}$ of $\mathbb{N}$ with
no one-to-one sequence $\left\langle x_{n}\right\rangle _{n=1}^{\infty}$
in $\mathbb{N}$ such that $PS\left(\left\langle x_{n}\right\rangle _{n=1}^{\infty}\right)\cup PP\left(\left\langle x_{n}\right\rangle _{n=1}^{\infty}\right)$
is contained in one cell of the partition $\mathcal{R}$.
\end{thm}

Given a discrete semigroup $(S,\cdot)$, we take the points of $\beta S$
to be the ultrafilters on $S$, identifying the principal ultrafilters
with the points of $S$ and thus pretending that $S\subseteq\beta S$.
Given $A\subseteq S$, let us define the subsets of $\beta S$ by
the following formula:
\[
\overline{A}=\{p\in\beta S:A\in p\}
\]
 Then the set $\left\{ \bar{A}\subseteq\beta S:A\subseteq\mathbb{N}\right\} $
forms a basis for the closed sets of $\beta S$ as well as for the
open sets. The operation $\cdot$ can be extended to the Stone-\v{C}ech
compactification $\beta S$ of $S$, so that $(\beta S,\cdot)$ is
a compact right topological semigroup (meaning that for any $p\in\beta S$,
the function $\rho_{p}:\beta S\rightarrow\beta S$ defined by $\rho_{p}(q)=q\cdot p$
is continuous) with $S$ contained in its topological center (meaning
that for any $x\in S$, the function $\lambda_{x}:\beta S\rightarrow\beta S$
defined by $\lambda_{x}(q)=x\cdot q$ is continuous). A nonempty subset
$I$ of a semigroup $T$ is called a left ideal of $S$ if $TI\subset I$,
a right ideal if $IT\subset I$, and a two sided ideal (or simply
an ideal) if it is both a left and right ideal. A minimal left ideal
is the left ideal that does not contain any proper left ideal. Similarly,
we can define minimal right ideal and smallest ideal.

Any compact Hausdorff right topological semigroup $T$ has a smallest
two sided ideal 
\begin{align*}
K(T) & =\bigcup\{L:L\text{ is a minimal left ideal }\text{of }T\}\\
 & =\bigcup\{R:R\text{ is a minimal right ideal of }T\}.
\end{align*}
 Given a minimal left ideal $L$ and a minimal right ideal $R$, $L\cap R$
is a group, and in particular contains an idempotent. An idempotent
in $K(T)$ is a minimal idempotent. If $p$ and $q$ are idempotents
in $T$ we write $p\leq q$ if and only if $p\cdot q=q\cdot p=p$.
An idempotent is minimal with respect to this relation if and only
if it is a member of the smallest ideal $K\left(T\right)$ of $T$.

Given $p,q\in\beta S$ and $A\subseteq S$, $A\in p\cdot q$ if and
only if $\{x\in S:x^{-1}A\in q\}\in p$, where $x^{-1}A=\{y\in S:x\cdot y\in A\}$.
See \cite{key-7} for an elementary introduction to the algebra of
$\beta S$ and for any unfamiliar details. 

\section{The Ring of Gaussian Integers $\mathbb{Z}\left[i\right]$}

Let $\left(S,+\right)$ is a commutative semigroup. A set $A\subseteq S$
is called an $\text{IP}^{*}$ set if it belongs to every idempotent
in $\beta S$ and it is called $\text{central}^{*}$ if it is contained
by every minimal idempotent. Bergelson and Hindman proved in \cite{key-2}
that $\text{IP}^{*}$ sets in $\mathbb{N}$ possess substantial amount
of additive and multiplicative properties. Given a sequence $\left\langle x_{n}\right\rangle _{n=1}^{\infty}$
in $\mathbb{N}$, we let $FP$$\left(\left\langle x_{n}\right\rangle _{n=1}^{\infty}\right)$
be the product analogue of finite sums.
\begin{defn}
Let $\left(S,+\right)$ be a semigroup and let $\left\langle x_{n}\right\rangle _{n=1}^{\infty}$
is a sequence in $S$. The sequence $\left\langle y_{n}\right\rangle _{n=1}^{\infty}$
is a sum subsystem of $\left\langle x_{n}\right\rangle _{n=1}^{\infty}$
if and only if there is a sequence $\left\langle H_{n}\right\rangle _{n=1}^{\infty}$
in $\mathcal{P}_{f}\left(\mathbb{N}\right)$ such that for every $n\in\mathbb{N}$,
max$H_{n}$ $<$ min$H_{n+1}$ and $y_{n}=\sum_{t\in H_{n}}x_{t}$.
Where $\mathcal{P}_{f}\left(\mathbb{N}\right)$ is the set of all
finite subsets of $\mathbb{N}$.
\end{defn}

The following Theorem \cite[Theorem 2.6]{key-2} shows that $\text{IP}^{*}$
sets in $\mathbb{N}$ have substantially rich multiplicative structures.
\begin{thm}
\label{Theorem 2.2}Let $\left\langle x_{n}\right\rangle _{n=1}^{\infty}$
in $\mathbb{N}$ and $A$ be an $IP^{*}$ set in $\left(\mathbb{N},+\right)$.
Then there exists a sum subsystem $\left\langle y_{n}\right\rangle _{n=1}^{\infty}$
of $\left\langle x_{n}\right\rangle _{n=1}^{\infty}$ such that $FS\left(\left\langle y_{n}\right\rangle _{n=1}^{\infty}\right)\cup FP\left(\left\langle y_{n}\right\rangle _{n=1}^{\infty}\right)\subseteq A$.
\end{thm}

In \cite{key-3} the above theorem was extended for $\text{central}^{*}$
sets where the collection of sequences were taken from a class of
sequences called minimal sequence. 
\begin{defn}
Let $\left(S,+\right)$ be a discrete commutative semigroup. A sequence
$\left\langle x_{n}\right\rangle _{n=1}^{\infty}$ in $S$ is called
a minimal sequence if 
\[
\bigcap_{n=1}^{\infty}\overline{FS\left\langle x_{n}\right\rangle _{n=m}^{\infty}}\cap K(\beta S)\neq\emptyset.
\]
\end{defn}

\begin{thm}
\label{Theorem 2.4}Let $\left\langle x_{n}\right\rangle _{n=1}^{\infty}$
be a minimal sequence in $\mathbb{N}$ and $A$ be a $\text{central}^{*}$
set in $\left(\mathbb{N},+\right)$. Then there exists a sum subsystem
$\left\langle y_{n}\right\rangle _{n=1}^{\infty}$ of $\left\langle x_{n}\right\rangle _{n=1}^{\infty}$
such that $FS\left(\left\langle y_{n}\right\rangle _{n=1}^{\infty}\right)\cup FP\left(\left\langle y_{n}\right\rangle _{n=1}^{\infty}\right)\subseteq A$.
\end{thm}

The original Central Sets Theorem was proved by Furstenberg in \cite[Theorem 8.1]{key-5}.
Most general version of Central Sets Theorem was established in \cite{key-4.1}.
We state it here only for the case of a commutative semigroup.
\begin{thm}
Let $\left(S,+\right)$ be a commutative semigroup and let $\mathcal{T}=\,^{\mathbb{N}}S$,
Let $C$ be a central subset of $S$. Then there exist functions $\alpha:\mathcal{P}_{f}\left(\mathcal{T}\right)\rightarrow S$
and $H:\mathcal{P}_{f}\left(\mathcal{T}\right)\rightarrow\mathcal{P}_{f}\left(\mathbb{N}\right)$
such that
\begin{enumerate}
\item let $F,G\in\mathcal{P}_{f}\left(\mathcal{T}\right)$ and $F\subsetneq G$,
then $\max H\left(F\right)<\min H\left(G\right)$ and
\item whenever $m\in\mathbb{N}$, $G_{1},G_{2},....,G_{m}\in\mathcal{P}_{f}\left(\mathcal{T}\right)$
such that $G_{1}\subsetneq G_{2}\subsetneq......\subsetneq G_{m}$
and for each $i\in\left\{ 1,2,....,m\right\} $, $f_{i}\in G_{i}$
one has 
\[
\sum_{i=1}^{m}\left(\alpha\left(G_{i}\right)+\sum_{t\in H\left(G_{i}\right)}f_{i}(t)\right)\in C.
\]
\end{enumerate}
\end{thm}

In \cite{key-4.1} the sets which satisfy the conclusion of the latest
Central Sets Theorem are called $C$-sets.
\begin{defn}
Let $\left(S,+\right)$ be a commutative semigroup and let $A\subseteq S$.
Then $A$ is called a $C$-set if and only if there exist functions
$\alpha:\mathcal{P}_{f}\left(^{\mathbb{N}}S\right)\rightarrow S$
and $H:\mathcal{P}_{f}\left(^{\mathbb{N}}S\right)\rightarrow\mathcal{P}_{f}\left(\mathbb{N}\right)$
such that 
\begin{enumerate}
\item let $F,G\in\mathcal{P}_{f}\left(^{\mathbb{N}}S\right)$ and $F\subsetneq G$,
then $\max H\left(F\right)<\min H\left(G\right)$,
\item whenever $m\in\mathbb{N}$, $G_{1},G_{2},....,G_{m}\in\mathcal{P}_{f}\left(^{\mathbb{N}}S\right)$
such that $G_{1}\subsetneq G_{2}\subsetneq......\subsetneq G_{r}$
and for each $i\in\left\{ 1,2,....,r\right\} $, $f_{i}\in G_{i}$
one has 
\[
\sum_{i=1}^{m}\left(\alpha\left(G_{i}\right)+\sum_{t\in H\left(G_{i}\right)}f_{i}\left(t\right)\right)\in A.
\]
\end{enumerate}
\end{defn}

We now recall some definition from \cite{key-4.1}.
\begin{defn}
Let $\left(S,+\right)$ be a commutative semigroup and let $\mathcal{T}=\,^{\mathbb{N}}S$
be the set of all sequences in $S$.
\begin{enumerate}
\item A subset $A$ of $S$ is said to be a $J$-set if for every $F\in\mathcal{P}_{f}\left(\mathcal{T}\right)$
there exist $a\in S$ and $H\in\mathcal{P}_{f}\left(\mathbb{N}\right)$
such that for all $f\in F$, $a+\sum_{t\in H}f\left(t\right)\in A$.
\item $J\left(S\right)=\left\{ p\in\beta S:\left(\forall A\in p\right)\left(A\text{ is a }J\text{-set}\right)\right\} .$
\end{enumerate}
\end{defn}

The alternative description of $C$-set and $J$-set follows from
the following results:
\begin{thm}
\label{Theorem 2.8}(i) Let $\left(S,+\right)$ be a discrete commutative
semigroup and $A$ be a subset of $S$. Then $A$ is a $J$-set if
and only if $J\left(S\right)\cap\overline{A}\neq\emptyset$.

(ii) Let $\left(S,+\right)$ be a commutative semigroup and let $\mathcal{T}=\,^{\mathbb{N}}S$
be the set of all sequences in $S$, and let $A\subseteq S$. Then
$A$ is a $C$-set if and only if there is an idempotent $p\in\overline{A}\cap J\left(S\right)$.
\end{thm}

We have already discussed $\text{IP}^{*}$ and $\text{central}^{*}$
sets before. Now we need to introduce the notion of $C^{*}$-set.
\begin{defn}
Let $\left(S,+\right)$ be a discrete commutative semigroup. A set
$A\subseteq S$ is said to be a $C^{*}$-set if it belongs to all
the idempotents of $J\left(S\right)$.
\end{defn}

It is clear from the definition of $C^{*}$-set that 
\[
\text{IP}^{*}\text{-set}\Rightarrow C^{*}\text{-set}\Rightarrow\text{central}^{*}\text{-set}.
\]

Its obvious that $C^{*}$-sets contain finite sum and product for
minimal sequences but not for any sequences. In \cite{key-4} author
shows that $C^{*}$-sets contain similar configuration of more general
form than minimal sequences.
\begin{defn}
Let $\left(S,+\right)$ be a discrete commutative semigroup. Then
a sequence $\left\langle x_{n}\right\rangle _{n=1}^{\infty}$ is said
to be an almost minimal sequence if 
\[
\bigcap_{m=1}^{\infty}\overline{FS\left(\left\langle x_{n}\right\rangle _{n=m}^{\infty}\right)}\cap J\left(S\right)\neq\emptyset.
\]
\end{defn}

The author in \cite{key-4} gave an example of an almost minimal sequence
in $\mathbb{Z}$ which is not minimal and proved:
\begin{thm}
Let $\left\langle x_{n}\right\rangle _{n=1}^{\infty}$ be an almost
minimal sequence and $A$ be a $C^{*}$-set in $\left(\mathbb{N},+\right)$.
Then there exists a sum subsystem $\left\langle y_{n}\right\rangle _{n=1}^{\infty}$
of $\left\langle x_{n}\right\rangle _{n=1}^{\infty}$ such that 
\[
FS\left(\left\langle y_{n}\right\rangle _{n=1}^{\infty}\right)\cup FP\left(\left\langle y_{n}\right\rangle _{n=1}^{\infty}\right)\subseteq A.
\]
\end{thm}

\begin{fact}
\label{Fact 2.12}Let $\left\langle x_{n}\right\rangle _{n=1}^{\infty}$be
an almost minimal sequence in $\mathbb{Z}$ which is not minimal.
Then $FS\left(\left\langle x_{n}\right\rangle _{n=1}^{\infty}\right)$
is a $J$-set in $\mathbb{Z}$ by \cite[Theorem 2.7]{key-4} and as
the product of $J$-sets is again a $J$-set in finite product space
by \cite[Theorem 2.11]{key-9} we have $FS\left(\left\langle x_{n}\right\rangle _{n=1}^{\infty}\right)\times FS\left(\left\langle x_{n}\right\rangle _{n=1}^{\infty}\right)$
is a $J$-set in $\mathbb{Z}\left[i\right]$ (as $\mathbb{Z}^{2}$
and $\mathbb{Z}\left[i\right]$ are isomorphic under addition). Now
suppose $\left\langle z_{m}\right\rangle _{m=1}^{\infty}$ in $\mathbb{Z}\left[i\right]$
as 
\[
z_{m}=\left\{ \begin{array}{c}
x_{\frac{m+1}{2}}\text{ where }m\text{ is odd}\\
ix_{\frac{m}{2}}\text{where }m\text{ is even}
\end{array}\right..
\]
 Then $\bigcap_{m=1}^{\infty}\overline{FS\left(\left\langle z_{n}\right\rangle _{n=m}^{\infty}\right)}\cap J\left(\mathbb{Z}\left[i\right]\right)\neq\emptyset$
as 
\[
FS\left(\left\langle z_{n}\right\rangle _{n=1}^{\infty}\right)\supseteq FS\left(\left\langle x_{n}\right\rangle _{n=1}^{\infty}\right)\times FS\left(\left\langle x_{n}\right\rangle _{n=1}^{\infty}\right)
\]
 and $\left\langle z_{m}\right\rangle _{m=1}^{\infty}$ provides an
almost minimal sequence in $\mathbb{Z}\left[i\right]$ which is not
minimal. We can also construct a minimal sequence in $\mathbb{Z}\left[i\right]$
with the help of a minimal sequence in $\mathbb{Z}$ by the similar
construction.
\end{fact}

\begin{thm}
\label{Theorem 2.13}In the semigroup $\left(\mathbb{Z}\left[i\right],+\right)$,
the following conditions are equivalent:
\begin{enumerate}
\item $\left\langle x_{n}\right\rangle _{n=1}^{\infty}$ is an almost minimal
sequence.
\item $FS\left(\left\langle x_{n}\right\rangle _{n=1}^{\infty}\right)$
is a $J$-set.
\item there is an idempotent in $\bigcap_{m=1}^{\infty}\overline{FS\left(\left\langle x_{n}\right\rangle _{n=m}^{\infty}\right)}\cap J\left(\mathbb{Z}\left[i\right]\right)\neq\emptyset$.
\end{enumerate}
\end{thm}

\begin{proof}
(1)$\Rightarrow$(2) is obvious from the definition. Now, for (2)$\Rightarrow$(3)
Let $FS\left(\left\langle x_{n}\right\rangle _{n=1}^{\infty}\right)$
be a J-set. Then by Theorem \ref{Theorem 2.8}, we have $\overline{FS\left(\left\langle x_{n}\right\rangle _{n=1}^{\infty}\right)}\cap J\left(\mathbb{Z}\left[i\right]\right)\neq\emptyset$.
We choose $p\in\overline{FS\left(\left\langle x_{n}\right\rangle _{n=1}^{\infty}\right)}\cap J\left(\mathbb{Z}\left[i\right]\right)$.
By \cite[Lemma 5.11]{key-7}, $\bigcap_{m=1}^{\infty}\overline{FS\left(\left\langle x_{n}\right\rangle _{n=m}^{\infty}\right)}$
is a subsemigroup of $\beta\mathbb{Z}\left[i\right]$. As a consequence
of \cite[Theorem 3.5]{key-4.1}, it follows that $J\left(\mathbb{Z}\left[i\right]\right)$
is a subsemigroup of $\beta\mathbb{Z}\left[i\right]$. Also, $J\left(\mathbb{Z}\left[i\right]\right)$
being closed, it is a compact subsemigroup of $\beta\mathbb{Z}\left[i\right]$.
Therefore, it suffices to show that for each $m$, $\bigcap_{m=1}^{\infty}\overline{FS\left(\left\langle x_{n}\right\rangle _{n=m}^{\infty}\right)}\cap J\left(\mathbb{Z}\left[i\right]\right)\neq\emptyset$,
because then it must contain an idempotent. For this, it in turn suffices
to show that, for any $m\in\mathbb{N}$ we have $\overline{FS\left(\left\langle x_{n}\right\rangle _{n=m}^{\infty}\right)}\cap J\left(\mathbb{Z}\left[i\right]\right)\neq\emptyset$.
So let $m\in\mathbb{N}$ be given. Then 
\begin{align*}
FS\left(\left\langle x_{n}\right\rangle _{n=1}^{\infty}\right) & =FS\left(\left\langle x_{n}\right\rangle _{n=1}^{m-1}\right)\cup FS\left(\left\langle x_{n}\right\rangle _{n=m}^{\infty}\right)\cup\\
 & \bigcup\left\{ t+FS\left(\left\langle x_{n}\right\rangle _{n=m}^{\infty}\right):t\in FS\left(\left\langle x_{n}\right\rangle _{n=1}^{m-1}\right)\right\} .
\end{align*}

Now, $FS\left(\left\langle x_{n}\right\rangle _{n=1}^{m-1}\right)\in p$
could not happen for $p\in\beta\mathbb{Z}\left[i\right]\setminus\mathbb{Z}\left[i\right]$
and the proof is obvious if $FS\left(\left\langle x_{n}\right\rangle _{n=m}^{\infty}\right)\in p$.
If the last term in the equation is in $p$, then for some $t\in FS\left(\left\langle x_{n}\right\rangle _{n=1}^{m-1}\right)$
and $q\in\overline{FS\left(\left\langle x_{n}\right\rangle _{n=m}^{\infty}\right)}$
so that $t+q=p$. For every $F\in q$, we have $t\in\left\{ z\in\mathbb{Z}\left[i\right]:-z+\left(t+F\right)\in q\right\} $
so that $t+F\in p$. Since $J$-sets in $\left(\mathbb{Z}\left[i\right],+\right)$
are translation invariant, $F$ becomes a $J$-set. Thus $q\in\bigcap_{m=1}^{\infty}\overline{FS\left(\left\langle x_{n}\right\rangle _{n=m}^{\infty}\right)}\cap J\left(\mathbb{Z}\left[i\right]\right)$.
\end{proof}
\begin{lem}
\label{Lemma 2.14} Let $\left\langle x_{n}\right\rangle _{n=1}^{\infty}$
be a sequence in $\mathbb{Z}\left[i\right]$ and let $0\neq z\in\mathbb{Z}\left[i\right]$.
Prove that for each $m\in\mathbb{N}$ there exists $H\in\mathcal{P}_{f}\left(\left\{ n\in\mathbb{N}:n>m\right\} \right)$
such that $z\mid\sum_{n\in H}x_{n}$. In other words, For any $0\neq z\in\mathbb{Z}\left[i\right]$,
the set $z\mathbb{Z}\left[i\right]$ is an $\text{IP}^{*}$ set in
$\mathbb{Z}\left[i\right]$.
\end{lem}

\begin{proof}
Fix $m\in\mathbb{N}$. For any $n>m$ and by division algorithm, $\exists\,w,r\in\mathbb{Z}\left[i\right]$
such that $x_{n}=z\cdot w+r,\,\text{where }0\leq\left|r\right|<\left|z\right|$
and there will be at most $\left|z\right|^{2}$ number of distinct
remainders $r$ for different $x_{n}$. If for some $x_{n}$, $r=0$
then $H=\left\{ n\right\} $. Otherwise, as the remainders are finite
at least one of the remainder $i\in\mathbb{Z}\left[i\right]$ repeated
infinitely. Take $\left|z\right|^{2}$ members of $\left\langle x_{n}\right\rangle _{n=m+1}^{\infty}$
those remainders are $i$. Take $H$ as the set of suffix of those
$\left|z\right|^{2}$ members.
\end{proof}
The author likes to thank Prof. Neil Hindman for providing the proof
of Lemma \ref{Lemma 2.18}(b), which is due to Prof. Dona Strauss.
\begin{lem}
\label{Lemma 2.15} Let $f:\mathbb{N}\rightarrow\mathbb{Z}\left[i\right]$
and let $\left\langle H_{n}\right\rangle _{n=1}^{\infty}$ be a sequence
in $\mathcal{P}_{f}\left(\mathbb{N}\right)$ and $z\in\mathbb{Z}\left[i\right]\setminus\left\{ 0\right\} $.
There is a union subsystem $\left\langle G_{n}\right\rangle _{n=1}^{\infty}$
of $\left\langle H_{n}\right\rangle _{n=1}^{\infty}$ such that for
each $n\in\mathbb{N},\:\sum_{t\in G_{n}}f\left(t\right)\in z\mathbb{Z}\left[i\right]$.
\end{lem}

\begin{proof}
Fix $z\in\mathbb{Z}\left[i\right]$. Let $\left\langle x_{n}\right\rangle _{n=1}^{\infty}$
be a sequence as $x_{n}=\sum_{t\in H_{n}}f\left(t\right)$. Let $m=1$,
then by lemma \ref{Lemma 2.14} we get $K_{1}\in\mathcal{P}_{f}\left(\left\{ n\in\mathbb{N}:n>m\right\} \right)$
such that $z\mid\sum_{t\in K_{1}}x_{t}$. Take $n_{1}=\max K_{1}$
and $m=n_{1}$, then again by lemma \ref{Lemma 2.14} we get $K_{2}\in\mathcal{P}_{f}\left(\left\{ n\in\mathbb{N}:n>m\right\} \right)$
such that $z\mid\sum_{t\in K_{2}}x_{t}$. Continuing this process
we get the sequences $\left\langle K_{n}\right\rangle _{n=1}^{\infty}$
such that for each $n\in\mathbb{N}$, $\max K_{n}<\min K_{n+1}$ \&
$\left\langle G_{n}\right\rangle _{n=1}^{\infty}$ as $G_{n}=\bigcup_{t\in K_{n}}H_{t}$.

Thus for each $n\in\mathbb{N}$, $z\mid\sum_{t\in K_{n}}x_{t}\,\Rightarrow\,z\mid\sum_{t\in K_{n}}\sum_{i\in H_{t}}f\left(i\right)\,\Rightarrow z\mid\sum_{t\in G_{n}}f\left(t\right)$.
This proves the lemma.
\end{proof}
\begin{lem}
\label{Lemma 2.16} Let $F\in\mathcal{P}_{f}\left(\mathbb{^{N}Z}\left[i\right]\right)$
and let $z\in\mathbb{Z}\left[i\right]\setminus\left\{ 0\right\} $.
There exists a sequence $\left\langle K_{n}\right\rangle _{n=1}^{\infty}$
in $\mathcal{P}_{f}\left(\mathbb{N}\right)$ such that for each $n\in\mathbb{N},\,\max K_{n}<\min K_{n+1}$
and for each $f\in F$ and each $n\in\mathbb{N}$, $\sum_{t\in K_{n}}f\left(t\right)\in z\mathbb{Z}\left[i\right]$.
\end{lem}

\begin{proof}
Let $F=\left\{ f_{1},f_{2},\ldots,f_{k}\right\} $ and pick a sequence
$\left\langle H_{n}\right\rangle _{n=1}^{\infty}$ in $\mathcal{P}_{f}\left(\mathbb{N}\right)$.
Then for $f_{1}$, by lemma \ref{Lemma 2.15} we get a union subsystem
$\left\langle H_{n}^{1}\right\rangle _{n=1}^{\infty}$ of $\left\langle H_{n}\right\rangle _{n=1}^{\infty}$
such that for each $n\in\mathbb{N},\:\sum_{t\in H_{n}^{1}}f_{1}\left(t\right)\in z\mathbb{Z}\left[i\right]$.
Now, for $f_{2}$, again using lemma \ref{Lemma 2.15} we get a union
subsystem $\left\langle H_{n}^{2}\right\rangle _{n=1}^{\infty}$ of
$\left\langle H_{n}^{1}\right\rangle _{n=1}^{\infty}$ such that for
each $n\in\mathbb{N},\:\sum_{t\in H_{n}^{2}}f_{2}\left(t\right)\in z\mathbb{Z}\left[i\right]$
and also for each $n\in\mathbb{N},\:\sum_{t\in H_{n}^{2}}f_{1}\left(t\right)\in z\mathbb{Z}\left[i\right]$.
Continuing this process $k$-times, we have a union subsystem $\left\langle H_{n}^{k}\right\rangle _{n=1}^{\infty}$
of $\left\langle H_{n}^{k-1}\right\rangle _{n=1}^{\infty}$, such
that for each $n\in\mathbb{N},\:\sum_{t\in H_{n}^{k}}f_{i}\left(t\right)\in z\mathbb{Z}\left[i\right],\,i\in\left\{ 1,2,\ldots,k\right\} $.
Take $K_{n}=H_{n}^{k}\ \forall n\in\mathbb{N}$, we have the lemma.
\end{proof}
\begin{lem}
\label{Lemma 2.17}Let $p$ be an idempotent in $\left(\beta\mathbb{Z}\left[i\right],+\right)$
and let $\alpha\in\mathbb{Q}\left[i\right]$. Then $\alpha\cdot p$
is also an idempotent in $\left(\beta\mathbb{Z}\left[i\right],+\right)$.
\end{lem}

\begin{proof}
The function $I_{\alpha}:\mathbb{Z}\left[i\right]\rightarrow\mathbb{Q}\left[i\right]$
defined by $I_{\alpha}\left(z\right)=\alpha\cdot z$ is an injective
homomorphism and so is its continuous extension $\widetilde{I_{\alpha}}:\beta\mathbb{Z}\left[i\right]\rightarrow\beta\mathbb{Q}\left[i\right]_{d}$
by \cite[Corollary 4.22]{key-7}. So $\alpha\cdot p$ is an idempotent
in $\left(\beta\mathbb{Q}\left[i\right]_{d},+\right)$. Furthermore,
if $\alpha=\frac{z_{1}}{z_{2}}$, with $z_{1}z_{2}\in\mathbb{Z}\left[i\right]$,
then as by lemma \ref{Lemma 2.14} $z_{2}\mathbb{Z}\left[i\right]\in p$
implies that $\mathbb{Z}\left[i\right]\in\alpha\cdot p$ and hence
that $\alpha\cdot p\in\beta\mathbb{Z}\left[i\right]$.
\end{proof}
\begin{lem}
\label{Lemma 2.18}(a) If $A$ is central set in $\left(\mathbb{Z}\left[i\right],+\right)$
then $zA$ and $z^{-1}A$ is also a central set for any $z\in\mathbb{Z}\left[i\right]\setminus\left\{ 0\right\} $.

(b) If $A$ is $C$-set in $\left(\mathbb{Z}\left[i\right],+\right)$
then $zA$ and $z^{-1}A$ is also a $C$-set for any $z\in\mathbb{Z}\left[i\right]\setminus\left\{ 0\right\} $.
\end{lem}

\begin{proof}
(a) Take $z\in\mathbb{Q}\left[i\right]$. Then $I_{z}:\mathbb{Z}\left[i\right]\rightarrow\mathbb{Q}\left[i\right]$
is a homomorphism and so its continuous extension $\widetilde{I_{z}}:\beta\mathbb{Z}\left[i\right]\rightarrow\beta\mathbb{Q}\left[i\right]_{d}$
by \cite[Corollary 4.22]{key-7}, where $\widetilde{I_{z}}\left(p\right)=z\cdot p$.
Thus if $p$ is an idempotent, then $z\cdot p$ is also an idempotent
by lemma \ref{Lemma 2.17} and $z\cdot p\in\widetilde{I_{z}}\left[K\left(\beta\mathbb{Z}\left[i\right]\right)\right]=K\left(\overline{z\mathbb{Z}\left[i\right]}\right)$.
Now assume $z=\frac{z_{1}}{z_{2}}$ then $z_{2}\mathbb{Z}\left[i\right]\subseteq z^{-1}\left(z_{1}\mathbb{Z}\left[i\right]\right)$
thus $z_{1}\mathbb{Z}\left[i\right]\in z\cdot p$ as $z_{2}\mathbb{Z}\left[i\right]\in p$
by \ref{Lemma 2.14}. So, $z\cdot p\in\beta\mathbb{Z}\left[i\right]$.
Also $z\cdot p\in K\left(\overline{z\mathbb{Z}\left[i\right]}\right)\cap\overline{z_{1}\mathbb{Z}\left[i\right]}$
and $z_{1}\mathbb{Z}\left[i\right]\subseteq z\mathbb{Z}\left[i\right]$
and thus $K\left(\overline{z_{1}\mathbb{Z}\left[i\right]}\right)=K\left(\overline{z\mathbb{Z}\left[i\right]}\right)\cap\overline{z\mathbb{Z}\left[i\right]}$
\cite[Theorem 1.65]{key-7}. Since every idempotent of $\beta\mathbb{Z}\left[i\right]$
is in $\overline{z_{1}\mathbb{Z}\left[i\right]}$, we have that $\overline{z_{1}\mathbb{Z}\left[i\right]}\cap K\left(\beta\mathbb{Z}\left[i\right]\right)\neq\emptyset$
and consequently $K\left(\overline{z_{1}\mathbb{Z}\left[i\right]}\right)=\overline{z_{1}\mathbb{Z}\left[i\right]}\cap K\left(\beta\mathbb{Z}\left[i\right]\right)$
\cite[Theorem 1.65]{key-7}. Thus $z\cdot p\in K\left(\beta\mathbb{Z}\left[i\right]\right)$.

(b) Pick an idempotent $p\in J\left(\mathbb{Z}\left[i\right]\right)\cap\bar{C}$.
Then $z\cdot p$ and $\frac{1}{z}\cdot p$ are idempotents in $\beta\mathbb{Z}\left[i\right]$.
We claim that $z\cdot p,\frac{1}{z}\cdot p\in J\left(\mathbb{Z}\left[i\right]\right)$,
so, first let $A\in z\cdot p$. We will show that $A$ is a $J$-set.
So, let $F\in\mathcal{P}_{f}\left(\mathbb{^{N}Z}\left[i\right]\right)$.
Let $\left\langle K_{n}\right\rangle _{n=1}^{\infty}$ be as guaranteed
by \ref{Lemma 2.16}. For each $f\in F$ and $n\in\mathbb{N}$, let
$g_{f}\left(n\right)=\frac{1}{z}\sum_{t\in K_{n}}f\left(t\right)$.
Now $z^{-1}A\in p$ so pick $a\in\mathbb{Z}\left[i\right]$ and $G\in\mathcal{P}_{f}\left(\mathbb{N}\right)$
such that for each $f\in F$, $a+\sum_{n\in G}g_{f}\left(n\right)\in d^{-1}A$.
Let $H=\bigcup_{n\in G}K_{n}$. Then for $f\in F$, $da+\sum_{t\in H}f\left(t\right)\in A$.

Again, let $A\in\frac{1}{z}\cdot p$. We will show that $A$ is $J$-set.
So let $F\in\mathcal{P}_{f}\left(\mathbb{^{N}Z}\left[i\right]\right)$.
For $f\in F$, define $g_{f}\in\mathbb{^{N}Z}\left[i\right]$ by for
each $n\in\mathbb{N},$ $g_{f}\left(n\right)=zf\left(n\right)$. Since,
$zA\in p$, $zA$ is a $J$-set, so pick $a\in\mathbb{Z}\left[i\right]$
and $H\in\mathcal{P}_{f}\left(\mathbb{N}\right)$ such that for each
$f\in F$, $a+\sum_{t\in H}g_{f}\left(t\right)\in zA$. Since, $zA\subseteq z\mathbb{Z}\left[i\right]$
and each $g_{f}\left(t\right)\in z\mathbb{Z}\left[i\right]$, we must
have $a\in z\mathbb{Z}\left[i\right]$ so $\frac{1}{z}a\in\mathbb{Z}\left[i\right]$
and for each $f\in F$, $\frac{1}{z}a+\sum_{t\in H}f\left(t\right)\in A$.
\end{proof}
\begin{lem}
\label{Lemma 2.19}(a) If $A$ is a $\text{central}^{*}$ set in $\left(\mathbb{Z}\left[i\right],+\right)$
then $z^{-1}A$ is also a $\text{central}^{*}$ set for any $z\in\mathbb{Z}\left[i\right]\setminus\left\{ 0\right\} $.

(b) If $A$ is a $C^{*}$-set in $\left(\mathbb{Z}\left[i\right],+\right)$
then $z^{-1}A$ is also a $C^{*}$-set for any $z\in\mathbb{Z}\left[i\right]\setminus\left\{ 0\right\} $.
\end{lem}

\begin{proof}
(a) Let $A$ be $\text{central}^{*}$ set and $z\in\mathbb{Z}\left[i\right]\setminus\left\{ 0\right\} $.
To prove that $z^{-1}A$ is a $\text{central}^{*}$ set it is sufficient
to show that for any central set $C$, $C\cap z^{-1}A\neq\emptyset$.
Since, $C$ is central, $zC$ is also central so that $A\cap zC\neq\emptyset$.
Choose $c\in A\cap zC$ and $k\in C$ such that $c=zk$. Therefore,
$k\in z^{-1}A$ and thus $C\cap z^{-1}A\neq\emptyset$. 

(b) The proof is analogous to (a).
\end{proof}
\begin{thm}
(a) Let $\left\langle x_{n}\right\rangle _{n=1}^{\infty}$ be a minimal
sequence and $A$ be a $\text{central}^{*}$ set in $\left(\mathbb{Z}\left[i\right],+\right)$.
Then there exists a sum subsystem $\left\langle y_{n}\right\rangle _{n=1}^{\infty}$
of $\left\langle x_{n}\right\rangle _{n=1}^{\infty}$ such that 
\[
FS\left(\left\langle y_{n}\right\rangle _{n=1}^{\infty}\right)\cup FP\left(\left\langle y_{n}\right\rangle _{n=1}^{\infty}\right)\subseteq A.
\]

(b) Let $\left\langle x_{n}\right\rangle _{n=1}^{\infty}$ be an almost
minimal sequence and $A$ be a $C^{*}$-set in $\left(\mathbb{Z}\left[i\right],+\right)$.
Then there exists a sum subsystem $\left\langle y_{n}\right\rangle _{n=1}^{\infty}$
of $\left\langle x_{n}\right\rangle _{n=1}^{\infty}$ such that 
\[
FS\left(\left\langle y_{n}\right\rangle _{n=1}^{\infty}\right)\cup FP\left(\left\langle y_{n}\right\rangle _{n=1}^{\infty}\right)\subseteq A.
\]
\end{thm}

\begin{proof}
(a) Since $\left\langle x_{n}\right\rangle _{n=1}^{\infty}$ is a
minimal sequence in $\mathbb{Z}\left[i\right]$ we can find some minimal
idempotent $p\in\beta\mathbb{Z}\left[i\right]$ for which $FS\left(\left\langle x_{n}\right\rangle _{n=1}^{\infty}\right)\in p$.
Again, since $A$ is a $\text{central}^{*}$ subset of $\mathbb{Z}\left[i\right]$,
by the lemma \ref{Lemma 2.19} for every $z\in\mathbb{Z}\left[i\right]$,
$z^{-1}A\in p$. Let $A^{\star}=\left\{ s\in A:-s+A\in p\right\} $.
Then by \cite[ Lemma 4.14]{key-7} $A^{\star}\in p$. Then we can
choose $y_{1}\in A^{\star}\cap FS\left(\left\langle x_{n}\right\rangle _{n=1}^{\infty}\right)$.
Inductively, let $m\in\mathbb{N}$ and $\left\langle y_{i}\right\rangle _{i=1}^{m}$,
$\left\langle H_{i}\right\rangle _{i=1}^{m}$ in $\mathcal{P}_{f}\left(\mathbb{N}\right)$
be chosen with the following properties:
\begin{enumerate}
\item for all $i\in\left\{ 1,2,\ldots,m-1\right\} $, $\max H_{i}<\min H_{i+1}$;
\item if $y_{i}=\sum_{t\in H_{i}}x_{t}$ then $\sum_{t\in H_{m+1}}x_{t}\in A^{\star}$
and $FP\left(\left\langle y_{i}\right\rangle _{i=1}^{m}\right)\subseteq A^{\star}$.
\end{enumerate}
We observe that 
\[
\left\{ \sum_{t\in H}x_{t}:H\in\mathcal{P}_{f}\left(\mathbb{Z}\left[i\right]\right),\min H>\max H_{m}\right\} \in p.
\]
 We can choose $H_{m+1}\in\mathcal{P}_{f}\left(\mathbb{N}\right)$
such that $\max H_{m}<\min H_{m+1}$, $\sum_{t\in H_{m+1}}x_{t}\in A^{\star}$
, $\sum_{t\in H_{m+1}}x_{t}\in-z+A^{\star}$. For every $z\in FS\left(\left\langle y_{i}\right\rangle _{i=1}^{m}\right)$
and $\sum_{t\in H_{m+1}}x_{t}\in z^{-1}A^{\star}$ for every $z\in FP\left(\left\langle y_{i}\right\rangle _{i=1}^{m}\right)$.
Putting $y_{m+1}=\sum_{t\in H_{m+1}}x_{t}$, it shows that the induction
can be continued and this eventually proves the theorem.

(b) Since $\left\langle x_{n}\right\rangle _{n=1}^{\infty}$ is an
almost minimal sequence, by Theorem \ref{Theorem 2.13}, we can find
some idempotent $p\in J\left(\mathbb{Z}\left[i\right]\right)$ such
that $FS\left(\left\langle x_{n}\right\rangle _{n=1}^{\infty}\right)\in p$,
for each $m\in\mathbb{Z}\left[i\right]$. Again, since $A$ is a $C^{*}$-set
in $\left(\mathbb{Z}\left[i\right],+\right)$, from Lemma \ref{Lemma 2.19},
it follows that, $s^{-1}A\in p$, for every $s\in\mathbb{Z}\left[i\right]$.
Let $A^{\star}=\left\{ s\in A:-s+A\in p\right\} $. Then by \cite[ Lemma 4.14]{key-7},
we have $A^{\star}\in p$. Then we can choose $y_{1}\in A^{\star}\cap FS\left(\left\langle x_{n}\right\rangle _{n=1}^{\infty}\right)$.
Inductively, let $m\in\mathbb{N}$ and $\left\langle y_{i}\right\rangle _{i=1}^{m}$,
$\left\langle H_{i}\right\rangle _{i=1}^{m}$ in $\mathcal{P}_{f}\left(\mathbb{N}\right)$
be chosen with the following properties:
\begin{enumerate}
\item for all $i\in\left\{ 1,2,\ldots,m-1\right\} $, $\max H_{i}<\min H_{i+1}$;
\item if $y_{i}=\sum_{t\in H_{i}}x_{t}$ then $\sum_{t\in H_{m}}x_{t}\in A^{\star}$
and $FP\left(\left\langle y_{i}\right\rangle _{i=1}^{m}\right)\subseteq A^{\star}$.
\end{enumerate}
We observe that 
\[
\left\{ \sum_{t\in H}x_{t}:H\in\mathcal{P}_{f}\left(\mathbb{Z}\left[i\right]\right),\min H>\max H_{m}\right\} \in p.
\]
 Let us set 
\[
B=\left\{ \sum_{t\in H}x_{t}:H\in\mathcal{P}_{f}\left(\mathbb{Z}\left[i\right]\right),\min H>\max H_{m}\right\} ,
\]
 $E_{1}=FS\left(\left\langle y_{i}\right\rangle _{i=1}^{m}\right)$
and $E_{2}=FP\left(\left\langle y_{i}\right\rangle _{i=1}^{m}\right)$.
Now consider 
\[
D=B\cap A^{\star}\cap\bigcap_{s\in E_{1}}\left(-s+A^{\star}\right)\cap\bigcap_{s\in E_{2}}\left(s^{-1}A^{\star}\right).
\]
 Then $D\in p$. Choose $y_{m+1}\in D$ and $H_{m+1}$ in $\mathcal{P}_{f}\left(\mathbb{N}\right)$
such that $\min H_{m+1}>\max H_{m}$. Putting $y_{m+1}=\sum_{t\in H_{m+1}}x_{t}$,
it shows that the induction can be continued and this eventually proves
the theorem.
\end{proof}

\section{The Ring of Integer Quaternions}

In this section we turn our attention to the ring of integer quaternions
(also known as Lipschitz quaternions or Lipschitz integers)
\[
L=\left\{ a+bi+cj+dk:a,b,c,d\in\mathbb{Z}\right\} ,
\]
 where $i^{2}=j^{2}=k^{2}=ijk=-1.$ In \cite{key-10} author extended
Theorem \ref{Theorem 2.2} for weak rings. So it is interesting to
check whether Theorem \ref{Theorem 2.4} is true for more general
rings other than $\mathbb{Z}$ or $\mathbb{Z}\left[i\right]$. We
examine it for ring of integer quaternions, which is a non commutative
ring. Let us note that under addition $L$ is isomorphic with $\left(\mathbb{Z}^{4},+\right)$. 

Let $\left\langle x_{n}\right\rangle _{n=1}^{\infty}$ be a minimal
sequence in $\left(\mathbb{Z},+\right)$. Now, suppose the sequence
$\left\langle w_{n}\right\rangle _{n=1}^{\infty}$ in $L$ as 
\[
w_{n}=\left\{ \begin{array}{cc}
x_{p+1} & \text{If }n=4p+1,p=0,1,\ldots\\
ix_{p+1} & \text{If }n=4p+2,p=0,1,\ldots\\
jx_{p+1} & \text{If }n=4p+3,p=0,1,\ldots\\
kx_{p} & \text{If }n=4p,p=1,2,\ldots
\end{array}\right..
\]
 Then, $\left\langle w_{n}\right\rangle _{n=1}^{\infty}$ is a minimal
sequence in $\left(L,+\right)$ by fact \ref{Fact 2.12} and $FS\left(\left\langle w_{n}\right\rangle _{n=1}^{\infty}\right)\supseteq FS\left(\left\langle x_{n}\right\rangle _{n=1}^{\infty}\right)\times FS\left(\left\langle x_{n}\right\rangle _{n=1}^{\infty}\right)\times FS\left(\left\langle x_{n}\right\rangle _{n=1}^{\infty}\right)\times FS\left(\left\langle x_{n}\right\rangle _{n=1}^{\infty}\right)$.
\begin{lem}
\label{Lemma 3.1}If $A$ is central set in $\left(L,+\right)$ then
$aA,\,Ab\text{ and }aAb$ is also a central set for any $a,b\in L\setminus\left\{ 0\right\} $.
\end{lem}

\begin{proof}
We will prove it for left multiplication as right multiplication follows
from a similar method and $aAb$ is central follows from combined
application of the previous two.

Let $A$ be any central set in $\left(L,+\right)$. Take the homomorphism
$\varphi:L\rightarrow L$ define as $\varphi\left(x\right)=a\cdot x$.
Let $F_{a}$ be the set of all possible remainders when any $y\in L$
is divided by $a$. Thus, by division algorithm, $F_{a}$ is finite
and $L=\bigcup_{r\in F_{a}}\left(-r+\varphi\left(L\right)\right)$.
Thus $\varphi\left(L\right)=a\cdot L$ is syndetic ( and thus piecewise
syndetic) and by \cite[Lemma 4.6 (I)]{key-1}, $\varphi\left(A\right)=aA$
is central.
\end{proof}
\begin{lem}
\label{Lemma 3.2}Let $A\subseteq L$, and let $a,b\in L$. 

(a) If $A$ is a $\text{central}^{*}$ set in $\left(L,+\right)$,
then $a^{-1}A$ is a $\text{central}^{*}$ set in $\left(L,+\right)$.

(b) If $A$ is a $\text{central}^{*}$ set in $\left(L,+\right)$,
then $Ab^{-1}$ is a $\text{central}^{*}$ set in $\left(L,+\right)$.

(c) If $A$ is a $\text{central}^{*}$ set in $\left(L,+\right)$,
then $a^{-1}Ab^{-1}$ is a $\text{central}^{*}$ set in $\left(L,+\right)$.
\end{lem}

\begin{proof}
The proof is analogous to \ref{Lemma 2.19}.
\end{proof}
The following useful definition is from \cite{key-7}.
\begin{defn}
\label{=00005BHS, Def. 16.36=00005D} Let $\left(S,\cdot\right)$
be a semigroup, let $\left\langle x_{n}\right\rangle _{n=1}^{\infty}$
be a sequence in $S$ and let $k\in\mathbb{N}$. Then $AP\left(\left\langle x_{n}\right\rangle _{n=1}^{k}\right)$
is the set of all products of terms of $\left\langle x_{n}\right\rangle _{n=1}^{k}$
in any order with no repetitions. Similarly $AP\left(\left\langle x_{n}\right\rangle _{n=1}^{\infty}\right)$
is the set of all possible products of finite terms of $\left\langle x_{n}\right\rangle _{n=1}^{\infty}$
in any order with no repetitions.
\end{defn}

\begin{thm}
\label{Theorem 3.4}Let $A$ be a $\text{central}^{*}$ set in $\left(L,+\right)$
and let $\left\langle y_{n}\right\rangle _{n=1}^{\infty}$ be any
minimal sequence in $\left(L,+\right)$. Then there exists a sum subsystem
$\left\langle x_{n}\right\rangle _{n=1}^{\infty}$ of $\left\langle y_{n}\right\rangle _{n=1}^{\infty}$
such that if $m\geq2,F\in\mathcal{P}_{f}\left(\mathbb{N}\right)$
with $\min F\geq m$ and $b\in AP\left(\left\langle x_{n}\right\rangle _{n=1}^{m-1}\right)$,
then $b\cdot\sum_{t\in F}x_{t}\in A$. In particular, 
\[
FS\left(\left\langle x_{n}\right\rangle _{n=1}^{\infty}\right)\cup\left\{ b\cdot x_{m}:m\geq2\text{ and }b\in AP\left(\left\langle x_{n}\right\rangle _{n=1}^{m-1}\right)\right\} \subseteq A.
\]
\end{thm}

\begin{proof}
Since $\left\langle y_{n}\right\rangle _{n=1}^{\infty}$ is a minimal
sequence in $\left(L,+\right)$ we can find some minimal idempotent
$p$ of $\left(\beta L,+\right)$ for which $FS\left(\left\langle x_{n}\right\rangle _{n=1}^{\infty}\right)\in p$.
Then by lemma \ref{Lemma 3.2}, we have for each $a\in S$ that $a^{-1}A$
is a $\text{central}^{*}$ set in $\left(L,+,\right)$ and hence is
in $p$. For any set $A$, let $A^{\star}=\left\{ s\in A:-s+A\in p\right\} $.

Take $B_{1}=FS\left(\left\langle y_{n}\right\rangle _{n=1}^{\infty}\right)$
and thus $B_{1}\in p$. Pick $x_{1}\in B_{1}^{\star}$ and $H_{1}\in\mathcal{P}_{f}\left(\mathbb{N}\right)$
such that $x_{1}=\sum_{t\in H_{1}}y_{t}$.

Inductively, let $n\in\mathbb{N}$ and assume that we have chosen
$\left\langle x_{i}\right\rangle _{i=1}^{n}$, $\left\langle H_{i}\right\rangle _{i=1}^{n}$
and $\left\langle B_{i}\right\rangle _{i=1}^{n}$ such that for each
$i\in\left\{ 1,2,\ldots,n\right\} $:
\begin{enumerate}
\item $x_{i}=\sum_{t\in H_{i}}y_{t}$,
\item if $i>1$, then $\min H_{i}>\max H_{i-1}$
\item $B_{i}\in p$,
\item if $\emptyset\neq F\subseteq\left\{ 1,2,\ldots,i\right\} $ and $m=\min F$,then
$\sum_{j\in F}x_{j}\in B_{m}^{\star}$ and
\item if $i>1$, then $B_{i}\subseteq\bigcap\left\{ a^{-1}A:a\in AP\left(\left\langle x_{t}\right\rangle _{t=1}^{i-1}\right)\right\} $.
\end{enumerate}
Only hypotheses (1), (3) and (4) apply at $n=1$ and they hold trivially.
Let $k=\max H_{n}+1$. By assumption $FS\left(\left\langle y_{t}\right\rangle _{t=k}^{\infty}\right)\in p$.
Again, we have by lemma \ref{Lemma 3.2}, that for each $a\in L$,
$a^{-1}A$ is an $\text{central}^{*}$ set in $\left(L,+\right)$
and is thus in $p$.

Now, for any $m\in\left\{ 1,2,\ldots,n\right\} $ let
\[
E_{m}=\left\{ \sum_{j\in F}x_{j}:\emptyset\neq F\subseteq\left\{ 1,2,\ldots,n\right\} \text{ and }m=\min F\right\} .
\]

By hypothesis (4) we have for each $m\in\left\{ 1,2,\ldots.n\right\} $
and each $a\in E_{m}$ that $a\in B_{m}^{\star}$ and hence, by \cite[Lemma 4.14]{key-7},
$-a+B_{m}^{\star}\in p$. Let 
\[
B_{m+1}=FS\left(\left\langle y_{t}\right\rangle _{t=k}^{\infty}\right)\cap\bigcap\left\{ a^{-1}A:a\in AP\left(\left\langle x_{t}\right\rangle _{t=1}^{n}\right)\right\} \cap\bigcap_{m=1}^{n}\bigcap_{a\in E_{m}}\left(-a+B_{m}^{\star}\right).
\]
 Then $B_{n+1}\in p$. Choose $x_{n+1}=B_{n+1}^{\star}$. Since $x_{n+1}\in FS\left(\left\langle y_{t}\right\rangle _{t=k}^{\infty}\right)$,
choose $H_{n+1}\in\mathcal{P}_{f}\left(\mathbb{N}\right)$ such that
$\min H_{n+1}\geq k$ and $x_{n+1}=\sum_{t\in H_{n+1}}y_{t}$. Then
hypotheses (1), (2), (3), and (5) are satisfied directly.

To verify hypothesis (4), let $\emptyset\neq F\subseteq\left\{ 1,2,\ldots,n+1\right\} $
and let $m=\min F$. If $n+1\notin F$ , the conclusion holds by hypothesis,
so assume that $n+1\in F$. If $F=\left\{ n+1\right\} $, then $\sum_{j\in F}x_{j}=x_{n+1}\in B_{n+1}^{\star}$,
so assume $F\neq\left\{ n+1\right\} $ and let $G=F\setminus\left\{ n+1\right\} $.
Let $a=\sum_{j\in F}x_{j}$. Then $a\in E_{m}$ so $x_{n+1}\in-a+B_{m}^{\star}$
so $\sum_{j\in F}x_{j}=a+x_{n+1}\in B_{m}^{\star}$ as required.

We thus have that $\left\langle x_{n}\right\rangle _{n=1}^{\infty}$
is a sum subsystem of $\left\langle y_{n}\right\rangle _{n=1}^{\infty}$.
To complete the proof, let $m\geq2$, let $F\in\mathcal{P}_{f}\left(\mathbb{N}\right)$
with $\min F\geq m$ and let $a\in AP\left(\left\langle x_{n}\right\rangle _{n=1}^{m-1}\right)$.
Then by hypotheses (4) and (5), $\sum_{j\in F}x_{j}\in B_{m}\subseteq a^{-1}A$
so that $a\cdot\sum_{t\in F}x_{t}\in A$.
\end{proof}
\begin{thm}
Let $A$ be a $central^{*}$ set in $\left(L,+\right)$. Also let
$\left\langle y_{n}\right\rangle _{n=1}^{\infty}$ be a minimal sequence
in $\left(L,+\right)$. Then there exists a sum subsystem $\left\langle x_{n}\right\rangle _{n=1}^{\infty}$
of $\left\langle y_{n}\right\rangle _{n=1}^{\infty}$ in $L$ such
that 
\[
FS\left(\left\langle x_{n}\right\rangle _{n=1}^{\infty}\right)\cup AP\left(\left\langle x_{n}\right\rangle _{n=1}^{\infty}\right)\subseteq A.
\]
\end{thm}

\begin{proof}
Modify the hypothesis (5) in the theorem \ref{Theorem 3.4} as: 

If $i>1$, then 
\begin{align*}
B_{i} & \subseteq\bigcap\left\{ a^{-1}A:a\in AP\left(\left\langle x_{t}\right\rangle _{t=1}^{i-1}\right)\right\} \cap\bigcap\left\{ Ab^{-1}:b\in AP\left(\left\langle x_{t}\right\rangle _{t=1}^{i-1}\right)\right\} \\
 & \bigcap\left\{ a^{-1}Ab^{-1}:a,b\in AP\left(\left\langle x_{t}\right\rangle _{t=1}^{i-1}\right)\right\} .
\end{align*}

Then replace the definition of $B_{n+1}$ with 
\begin{align*}
B_{n+1} & =FS\left(\left\langle y_{t}\right\rangle _{t=k}^{\infty}\right)\cap\bigcap\left\{ a^{-1}A:a\in AP\left(\left\langle x_{t}\right\rangle _{t=1}^{i-1}\right)\right\} \\
 & \cap\bigcap\left\{ Ab^{-1}:b\in AP\left(\left\langle x_{t}\right\rangle _{t=1}^{i-1}\right)\right\} \cap\bigcap\left\{ a^{-1}Ab^{-1}:a,b\in AP\left(\left\langle x_{t}\right\rangle _{t=1}^{i-1}\right)\right\} \\
 & \cap\bigcap_{m=1}^{n}\bigcap_{a\in E_{m}}\left(-a+B_{m}^{\star}\right).
\end{align*}

This proves the theorem.
\end{proof}

\end{document}